\documentclass[review]{elsarticle}
\usepackage{amsmath,amssymb,amsthm}
\usepackage{multirow}
\usepackage{lineno,hyperref}
\usepackage{bm}
\usepackage{diagbox}
\usepackage{multirow}
\usepackage{tabularx}

\def\bR{\mathbb{R}}

\newtheorem{theorem}{Theorem}[section]
\newtheorem{lemma}[theorem]{Lemma}

\newtheorem{remark}{Remark}
\newtheorem{definition}[theorem]{Definition}
\modulolinenumbers[5]
\journal{****}
\begin{document}

\begin{frontmatter}
\title{\Large Variable metric extrapolation proximal iterative hard thresholding method for $\ell_0$ minimization problem}

\author[a]{Xue Zhang}
\ead{zhangxue2100@sjtu.edu.cn}
\author[b,c]{Xiaoqun Zhang}
\ead{xqzhang@sjtu.edu.cn}

\address[a]{School of mathematics and computer science, Shanxi Normal Univeristy, Shanxi 041001, CHINA}
\address[b]{Department of Mathematics, Shanghai Jiao Tong University, Shanghai 200240, CHINA}
\address[c]{Institute of Natural Sciences, Shanghai Jiao Tong University, Shanghai 200240, CHINA}
\begin{abstract}
In this paper, we consider the $\ell_0$ minimization problem whose objective function is the sum of $\ell_0$-norm and convex differentiable function. A variable metric type method which combines the PIHT method and the skill in quasi-newton method, named variable metric extrapolation proximal iterative hard thresholding (VMEPIHT) method, is proposed. Then we analyze its convergence, linear convergence rate and superlinear convergence rate  under appropriate assumptions. Finally, we conduct numerical experiments on compressive sensing problem and CT image reconstruction problem to confirm VMPIHT method's efficiency, compared with other state-of-art methods.
\end{abstract}
\end{frontmatter}

\section{Introduction}
In this paper, we consider the following $\ell_0$ minimization problem  proposed in compressed sensing (CS)
\begin{equation}\label{eqh}
\min_{x} H(x):=f(x)+\lambda \|x\|_0
\end{equation}
where $f(x)$ is convex, $\nabla f(x)$ is $L$-Lipschitz continuous and $\|x\|_0$ is the number
of nonzero components of $x$. Over the last decade, the sparse model with $\ell_0$-norm has been pursued in signal processing and
many other fields, such as image processing\cite{ZhangDong13,2019Wavelet},  machine learning\cite{Machine}, CT image reconstruction\cite{Zengli2018}, computer
vision\cite{2013Sparse}, signal recovery\cite{2020Iterative}, and
electrical capacitance tomography\cite{2017A}.

Despite problem \eqref{eqh} is NP-hard, there exist some methods to approximately solve the problem. To the best of our knowledge, there are three main schemes. The first scheme
is  the forward-backward(FB) method\cite{FB79,1979Ergodic} which originally proposed for solving the the sum of two monotone operators. For the following general problem
\begin{equation}\label{fplusg}
\min_{x} f(x)+g(x)
\end{equation}
whose special case is the problem \eqref{eqh} we consider, \cite{2013Convergence} provides a new convergence results for inexact forward-backward splitting method using Kurdyka-{\L}ojasiewicz inequality under nonconvex assumption.
The proximal iterative hard thresholding(PIHT) method proposed by \cite{Lu12} is to solve the problem \eqref{eqh}, whose recursive formula is
\begin{equation}\label{eqn:PIHT1}
x_{k+1}\in {\arg\min}_{x\in \bR^n}\lambda\|x\|_0+\frac{L}{2}\|x-x_k+\nabla f(x_k)/L\|^2+\frac{\mu}{2}\|x-x_k\|^2.
\end{equation}
For any starting point, $\{x_k\}$ generated by PIHT method converges to a local minimizer under the assumption that $f(x)$ is convex and $\nabla f(x)$ is $L$-Lipschitz continuous. In fact, PIHT method is a direct application of forward-backward splitting method on problem \eqref{eqh}.

The second scheme is extrapolation type methods. The inertial forward-backward (IFB) method \cite{bot2014inertial} for
 solving problem (\eqref{fplusg}) uses Bregman distance to replace Euclidean distance. If using Euclidean distance, the iterative scheme of IFB method is
\begin{eqnarray*}
&&y_{k+1}=x_{k}+2\beta_k(x_k-x_{k-1})\\
&&x_{k+1}\in \arg\min_{x \in X}g(x)+\frac{1}{4\alpha_k}\|x-x_{k}+2\alpha_k\nabla f(x_{k})\|^2+\frac{1}{4\alpha_k}\|x-y_{k+1}\|^2
\end{eqnarray*}
where $\alpha_k,\beta_k>0$ satisfy
\begin{eqnarray}
&&0<\underline{\alpha}\leq\alpha_k\leq \overline{\alpha},0<\beta_k\leq\beta,\;\; 1>\overline{\alpha}L+2\beta\frac{\overline{\alpha}}{\underline{\alpha}} \label{IFB1}
\end{eqnarray}
for some $\overline{\alpha},\underline{\alpha},\beta>0$.
Inequalities \eqref{IFB1} imply that  one  cannot have both of large $\alpha_k$ and $\beta_k$. \cite{2016A} proposed a multi-step IFB method which has two extrapolation step. And the parameters need to satisfy similar conditions. Monotone accelerated proximal gradient (mAPG) method\cite{chubulai} for solving nonconvex problem \eqref{fplusg} has the iterative scheme as follows
\begin{equation}\label{mAPG}
\begin{array}{l}
y_k=x_k+\dfrac{t_{k-1}}{t_k}(z_k-x_k)+\frac{t_{k-1}-1}{t_k}(x_k-x_{k-1}),\;\;t_{k+1}=\dfrac{\sqrt{1+4t_k^2}+1}{2}\\
z_{k+1}=\arg\min_zg(z)+\dfrac{1}{2\alpha_y}\|z-y_k+\alpha_y\nabla f(y_k)\|^2\\
v_{k+1}=\arg\min_vg(v)+\dfrac{1}{2\alpha_x}\|v-x_k+\alpha_x\nabla f(x_k)\|^2\\
x_{k+1}=
\left\{
\begin{array}{cl}
  z_{k+1},& \mbox{ if }f(z_{k+1})+g(z_{k+1})\leq f(v_{k+1})+g(v_{k+1}) \\
  v_{k+1},& \mbox{ otherwise}
\end{array}
\right.
\end{array}
\end{equation}
mAPG has $O(1/k^2)$~convergence rate if $f(x)$ and  $g(x)$ are convex. \cite{chubulai} also proposed a non-monotone APG(nmAPG) for saving the computation cost in each step. nmAPG method gets the next iteration point by $x^{k+1}=z^{k+1}$ when $f(z^{k+1})+g(z^{k+1})\leq c_k-\delta\|z^{k+1}-y^k\|^2$ where
\begin{equation}\label{nmAPG}
q_1=1,c_1=f(x_1)+g(x_1);q_{k+1}=\eta q_k+1, c_{k+1}=\frac{\eta q_kc_k+f(x^{k+1})+g(x^{k+1})}{q_{k+1}};
\end{equation}
otherwise, it gets the next iteration point same with mAPG.
Nonconvex inexact
accelerated proximal gradient (niAPG) method provided by \cite{2018Inexact} requires only
one proximal step (may be inexact) in each iteration, thus less expensive than mAPG. It's iterative scheme  is
\begin{equation}\label{niAPG}
\begin{array}{l}
y_k=x_k+\dfrac{k-1}{k}(x_k-x_{k-1}),\;\;\triangle_k=\max_{t=\max\{1,k-q\},\ldots,k}f(x_t)+g(x_t)\\
v_{k}=
\left\{
\begin{array}{cl}
  y_{k},& \mbox{ if }f(y_{k})+g(y_{k})\leq \triangle_k \\
  x_{k},& \mbox{ otherwise}
\end{array}
\right.\\
x_{k+1}=\arg\min_vg(v)+\dfrac{1}{2\eta}\|v-v_k+\eta\nabla f(v_k)\|^2\\
\end{array}
\end{equation}
\cite{2019A} proposed
a new proximal iterative hard thresholding method (nPIHT) for problem \eqref{eqh} with a convex constraint set $X$ whose iterative scheme  is
\begin{equation}\label{nPIHT}
\begin{array}{l}
y_{k+1}=x_{k}+|\mbox{sign}(x_k)|\omega_k(x_{k}-x_{k-1})\\
\mbox{if} \langle y_{k+1}-x_{k}, \nabla f(y_{k+1})\rangle>0 \mbox{ or } y_{k+1}\notin X \mbox{ reset} y_{k+1}=x_{k}\\
x_{k+1}\in \arg\min_{x \in X}\lambda\| x\|_0+\frac{L}{2}\|x-y_{k+1}+\frac{1}{L}\nabla f(y_{k+1})\|^2+\frac{\mu}{2}\|x-y_{k+1}\|^2.
\end{array}
\end{equation}
The nPIHT method has better numerical results than mAPG for $\ell_0$ minimization problem in \cite{2019A}. Table \eqref{e1} summarizes  some  differences of the above mentioned algorithms.

\begin{table}
\centering
\begin{tabular}{|c|c|c|}
\hline
Method& Assumption &  Convergence \\
\hline
IFB& nonconvex differentiable $f$, nonconvex $g$, KL & globally \\
mAPG& nonconvex differentiable $f$, nonconvex $g$& subsequence\\
niAPG& nonconvex differentiable $f$, nonconvex $g$& subsequence\\
nPIHT& convex differentiable $f$, $g(x)=\lambda\|x\|_0+\delta_X(x)$&globally \\\hline
\end{tabular}
\caption{Comparisons of extrapolation type methods. }\label{e1}
\end{table}

The third scheme is variable metric type methods usually using the following operator
\[
\mbox{prox}_g^A(x):=\arg\min_y g(y)+\frac{1}{2}\|y-x\|_A^2
\]
where $A$ is a symmetric positive definite matrice. Variable metric technique is usually used for accelerating the convergence of the FB method. The general iteration of a variable metric FB method \cite{BONETTINI2016A,2015Splitting,2016On,2016The,2013A,2020Variable,2016Unifying}can be stated as follows
\begin{eqnarray*}
&&y_{k}=\mbox{prox}_g^{\gamma_k^{-1}A_k}(x_{k}-\gamma_kA_k^{-1}\nabla f(x_{k}))\\
&&x_{k+1}=x_k+\lambda_k(y_{k}-x_{k}).
\end{eqnarray*}
Different choices of $A_{k}$ and $\gamma_k$ lead to different convergence properties and practical performances. Usually, a good metric matrix selection rule should be able to preserve the theoretical convergence of the iterates to a solution, spend less effort in the computation of  $\mbox{prox}_g^{\gamma_k^{-1}A_k}$ and have good numerical effect. Under suitable strategies for $A_{k}$, \cite{2009A,2015A,2014New} have numerically shown that variable metric type methods can reach performance comparable with algorithms who have superlinear convergence rate.
Several authors have analyzed the convergence of variable metric FB method when $f(x)$ is differentiable and $g(x)$ is nonconvex \cite{2015Splitting,2013A,2016Unifying}. The convergence can be proved when $\{A_k\}$ is any sequence of symmetric positive definite matrices and satisfies one of the following conditions
\begin{itemize}
\item $\gamma_kA_k\geq aI>LI;$
\item $0\prec mI\preceq A_k\preceq MI$, and $f(x)\leq f(x_k)+\langle x-x_k,\nabla f(x_k)\rangle+\frac{1}{2\gamma_k}\|x-x_k\|_{A_k}$.
\end{itemize}
For some specific issues, \cite{2015Splitting,2013A,2016Unifying} explained how to select metric matrix $\{A_k\}$  in detail.
But in general, the strategy of selecting metric matrix $\{A_k\}$ to ensure the theoretical convergence and improves the effectiveness of the whole method need further study.

For nonconvex and nonsmooth problem \eqref{fplusg}, it is very difficult to propose a convergent method  due to the reduction of convex properties. But $\|x\|_0$ has its own characteristics which is helpful to construct fast convergent method. So we focus on the special case, namely problem \eqref{eqh}. And we aim to design fast method using $\|x\|_0$'s characteristics. In fact, if $k$ is sufficiently large, the non-zero index set of iteration point $x_k$ remains unchanged \cite{Lu12,2019A}. Then  the methods\cite{Lu12,2019A} find the solution of $\min_{x\in C} f(x)$ where $C$ is a linear subspace instead of problem \eqref{eqh}. On the other side, variable metric type method can reach fast performance in practice. Encouraged by these, we proposed a variable metric extrapolation PIHT method which use PIHT method to get a reliable linear subspace $C$, then use variable metric method to minimization $f(x)$ fastly in subspace $C$. Our method is described in Section 2. Its convergence behavior is shown in Section 3. For a general convex differentiable function $f(x)$, any cluster point of iterative sequence is a local minimizer of $H(x)$. In particular, linear convergence rate and superlinear convergence rate are studied for quadratic function $f(x)$ which often happens in signal and image processing. If $f(x)$ is furthermore twice differentiable, sequence convergence is studied. The numerical assessment of the proposed method is described in Section 4, which confirm our method's efficiency compared with other state-of-art methods.

\section{ Algorithm}
\label{sec:alg}
\subsection{Preliminaries}
We first introduce some notations, concepts and results that will be used in this paper. For any $x\in\bR^n$, $x_i$ represents $x$'s $i$-th component and
$\|x\|_0$ denotes the number of $x$'s nonzero elements .
We define the zero element index set of a vector $x\in\mathbb{R}^n$ as $I(x):=\{i:x_i=0\}$. And for any index set $I\subseteq \{1,2,\ldots,n\}$,
\[
C_I:=\{x\in\bR^n: x_i=0\text{ for all }i\in I\}.
\]
 The projection operator defined on a set $C\subseteq \bR^n$ is denoted by
    \[
    P_{C}(x)=\arg\min_{z\in C}\frac{1}{2}\|z-x\|^2.
    \]
    $P_{C}(\cdot)$ is continuous, namely, for any convergent sequence $\{x^k\}$, it holds
 $\lim_{k\rightarrow +\infty}P_{C}(x^k)=P_{C}(\lim_{k\rightarrow +\infty}x^k).$
For any symmetric $A$ and $B$, $A\preceq B$ implies $B-A$ is a symmetric positive semidefinite matrix.

\smallskip
\begin{definition}
 A mapping $T:\bR^n \rightarrow\bR^n$ is said to be $L_T$-Lipschitz
continuous on the set $X\subseteq\bR^n$ if there exists $L_T>0$ such that
\[
\|T(x)-T(y)\| \leq L_T\|x-y\|,\quad \forall x, y\in X.
\]
\end{definition}

\smallskip
\begin{lemma}\cite{BT09}\label{Lips}
For $f: \bR^n\rightarrow \bR $, if $\nabla f(x)$ is L-Lipschitz continuous, the following inequality holds
\[
f(x)-f(y)\leq
\langle\nabla f(y),x-y\rangle+\frac{L}{2}\|x-y\|^2, \;\; \forall x,y \in \bR^n.
\]
\end{lemma}

\smallskip
\begin{lemma}\label{Q}
Let $D\in R^{m\times m}$ be a symmetric positive semidefinite matrices. Then $a^TD^2a\geq \lambda_{t} a^TDa$ for any vector $a\in R^m$, where $\rho$ is the smallest eigenvalue except 0.
\end{lemma}
\begin{proof}
Without loss of generality, let $\lambda_1\geq\lambda_2\geq\cdots\geq\lambda_t\geq\cdots\geq\lambda_m$ be $D$'s eigenvalue where $\lambda_{t+1}=\cdots=\lambda_m=0$. $v_1,v_2,\cdots,v_t,\cdots,v_m$ are eigenvectors correspondingly which are orthogonal.
Assume that $a=c_1v_1+\cdots+c_mv_m$. Then
\[
a^TD^2a=\sum_{i=1}^m\lambda_i^2c_i^2\|v_i\|^2=\sum_{i=1}^t\lambda_i^2c_i^2\|v_i\|^2\geq \sum_{i=1}^t\lambda_t\lambda_ic_i^2\|v_i\|^2=\lambda_t a^TDa.
\]
\end{proof}
\subsection{VMEPIHT method}
Emphasizing again,   we focus on the problem \eqref{eqh}, namely
\[
\min_{x} H(x):=f(x)+\lambda \|x\|_0.
\]
Throughout this paper, unless otherwise specified, our  assumptions on problem \eqref{eqh} are\\
\textbf{Assumption A:}
\begin{enumerate}
  \item $f$ is convex differentiable and  bounded from below;
  \item $\nabla f$ is $L$-Lipschitz continuous.
\end{enumerate}
We propose the following \emph{variable metric extrapolation PIHT method}, called VMEPIHT.

\vskip 0.5cm\hrule\vskip 0.1cm \noindent\textbf { VMEPIHT method}\vskip 0.1cm\hrule\vskip 0.2cm

\noindent Choose parameters $\mu>0, \lambda>0$; choose starting point $y^0$;  let $k=0$.

\noindent\textbf {while } the stopping criterion does not hold

Let
\begin{equation}\label{EPIHT}
x_{k}\in \arg\min_{x }\lambda\| x\|_0+\frac{L}{2}\|x-y_{k}+\frac{1}{L}\nabla f(y_{k})\|^2+\frac{\mu}{2}\|x-y_{k}\|^2
\end{equation}

compute positive definite matrix $H_{k}$ by some method, then
\[
y_{k+1}=P_{C_{I(x_{k})}}(x_{k}-\alpha_{k}H_{k}\nabla f(x_{k}))
\]
where step length $\alpha_k$ can guarantee the function value of $f$ to decrease.

$k=k+1$

\noindent\textbf {end(while)} \vskip 0.2cm\hrule\vskip 0.5cm

During the iteration, we use PIHT method  to get a reliable linear subspace $C_{I(x^{k})}$, then minimization $f(x)$ in linear subspace to get fast decrease of function value. That means $\|y_{k+1}\|_0\leq\|x_{k}\|_0$ and $f(y_{k+1})\leq f(x_k)$. Different choices of $H_{k}$ lead to different convergence properties and practical performances. The practical suitable strategies for selecting metric matrix $H_{k}$ and step length $\alpha_k$ will be declared in next section.
\begin{remark}
The solution of the equation \eqref{EPIHT} is given by
 $$x_{k}\in\mathcal{H}_{\sqrt{\frac{2{\lambda}}{L+\mu}}} (y_k-\frac{1}{L+\mu}\nabla f(y_k)),$$
 where $\mathcal{H}_{\sqrt{\frac{2 {\lambda}}{L+\mu}}}(\cdot)$ is the hard thresholding operator defined as
 \begin{equation}\label{eqn:ht}
(\mathcal{H}_{{\gamma}}(c))_i=
\left\{
\begin{array}{ll}
\{c_i\},      &\mbox{if}\; |c_i|>\gamma\\
\{0,c_i\} &\mbox{if}\; |c_i|=\gamma\\
\{0\},        &\mbox{if}\; |c_i|< \gamma.
\end{array}
\right.
\end{equation}
\end{remark}
\section{convergence}
In this section, under suitable assumptions, we analyze  the convergence behavior of VMEPIHT method. For simplicity, denote $C_*=C_{I(x^*)}$ and $C_k=C_{I(x^k)}$. To simplify proof of our main results, we first give the following lemmas.
\begin{lemma}\cite{2015Zhang}\label{HH}
\textbf{ (Continuity of $\mathcal{H_{\gamma}(\cdot)}$)} If $\gamma_k\rightarrow\gamma$, $x_k\rightarrow x$, $y_k\rightarrow y$ and
$x_k\in\mathcal{H}_{\gamma_k}(y_k)$, then $x\in\mathcal{H}_{\gamma}(y)$ holds.
\end{lemma}

\begin{lemma}\label{lem:nochange}
Let $\{x_k\}$ be the sequence generated by VMEPIHT method. Then
\begin{enumerate}
\item $\{H(x_k)\}$ is non-increasing;
  \item $\sum_{k=1}^\infty\|x_k-y_{k}\|^2<\infty$, $\|x_k-y_{k}\|^2\rightarrow0$;
  \item $I(x_k)$ changes only finitely often.
\end{enumerate}
\end{lemma}
\begin{proof}
1. Since $\nabla f(x)$ is $L$-Lipschitz continuous, from Lemma \eqref{Lips}, we have
\begin{equation}\label{Lip}
f(x_{k})-f(y_{k})\leq
\langle\nabla f(y_{k}),x_{k}-y_{k}\rangle+\frac{L}{2}\|x_{k}-y_{k}\|^2.
\end{equation}
Meanwhile, it follows from VMEPIHT method that
\[
\lambda\|x_{k}\|_0+
\frac{L}{2}\|x_{k}-y_{k}+\frac{\nabla f(y_{k})}{L}\|_2+\frac{\mu}{2}\|x_{k}-y_{k}\|^2
\leq \lambda\|y_{k}\|_0+\frac{L}{2}\|\frac{\nabla f(y_{k})}{L}\|^2
\]
\[
\|y_{k}\|_0\leq\|x_{k-1}\|_0,\;\; f(y_{k})\leq f(x_{k-1})
\]
Summing up the above four inequalities yields
\begin{equation}\label{e3}
H(x_{k})+\frac{\mu}{2}\|x_{k}-y_{k}\|_2\leq H(x_{k-1}).
\end{equation}
It's obvious that $\{H(x_k)\}$ is non-increasing.

2. Summing up the inequality \eqref{e3} over $k=1,\ldots,n$, we have
\[
H(x^n)+\frac{\mu}{2}\sum_{k=1}^n\|x^k-y^{k}\|^2\leq H(x^0).
\]
So $\{\sum_{k=1}^n\|x^k-y^{k}\|^2\}$ has upper bound since $H$ has lower bound. Then we can conclude
\[
\sum_{k=1}^\infty\|x^k-y^{k}\|^2<\infty,\;\; \|x^k-y^{k}\|\rightarrow 0.
\]

3. From the subproblem (\eqref{EPIHT}) and the hard thresholding operator \eqref{eqn:ht}, we have $|(x_k)_i|\geq \sqrt{2\lambda/(L+\mu)}$ for any $i\not\in I(x_k)$. Hence, we have $\|x_k-y_{k}\|\geq \sqrt{2\lambda/(L+\mu)}>0$ if $I(x_k)^C\nsubseteq I(y_{k})^C$ where $I(x_k)^C$ denotes supplementary set of $I(x_k)$ in $\{1,2,\cdots,n\}$. Combining $\|x_k-y_{k}\|\rightarrow 0$ and $I(y_k)^C\subseteq I(x_{k-1})^C$, it is easy to see that $I(x_k)^C\subseteq I(x_{k-1})^C$ always hold if $k$ is sufficiently large. Then
$I(x_k)$ must change only finitely often.
\end{proof}

\begin{theorem}\label{thm:kexists}
Let $H(x)$ be the objective function, and $\{x_k\}$ be the sequence generated by VMEPIHT method, then
\begin{enumerate}
  \item any cluster point of $\{x_k\}$ is a local minimizer of $H(x)$;
      \item $H(x_k)\rightarrow H(x^*)$ where $x^*$ is a cluster point of $\{x_k\}$.
\end{enumerate}
\end{theorem}
\begin{proof}
1. Assume that $x^*$ is a cluster point of $\{x_k\}$ and the subsequence
$\{x_{k_j}\}$ converging to $x^*$. It follows from  Lemma \eqref{lem:nochange} that $\|x_k-y_{k}\|\rightarrow 0$ and then $y_{k_j}\rightarrow x^*$ right now.
The iterative formula \eqref{EPIHT} of VMEPIHT method is equivalent to
 $$x_{k}\in\mathcal{H}_{\sqrt{\frac{2{\lambda}}{L+\mu}}} (y_k-\frac{1}{L+\mu}\nabla f(y_k)).$$
 Letting  $k$ be equal to $k_j$ and $j$ tends to infty, together with the Lemma \eqref{HH}, we obtain
\begin{eqnarray*}
&x^*=\mathcal{H}_{\sqrt{\frac{2{\lambda}}{L+\mu}}}(x^*-\frac{\nabla f(x^*)}{L+\mu}).
\end{eqnarray*}
In particular $x^*_i=x^*_i-\frac{(\nabla f(x^*))_i}{L+\mu},\;\;i\not\in I(x^*)$, namely  $(\nabla f(x^*))_i=0,\;\;i\not\in I(x^*)$.

Denote
\[
 U:=\{\Delta x\in\bR^n:\;\;\|\Delta x\|_{\infty}<\frac{1}{2}\min_{i\in I(x^*)}\{\frac{\lambda}{|(\nabla f(x^*))_i|},1\};\;\; |\Delta x_i|<|x^*_i|,\forall i\notin I(x^*) \},
 \]
It is clear that $x^*+U$ is a neighborhood of $x^*$. For  any $\Delta x\in U$,  we have \begin{itemize}
           \item
           $\lambda\|x^*_i+
\Delta x_i\|_0=\lambda\|x^*_i\|_0,\;\forall i\notin I(x^*)$ since $|\Delta x_i|<|x^*_i|$;
           \item $\Delta x_i(\nabla f(x^*))_i=0,\;\forall i\notin I(x^*)$ ;
           \item $\lambda\|\Delta x_i\|_0+(\nabla f(x^*))_i\Delta x_i\geq0,\; i\in I(x^*)$ since $\|\Delta x\|_{\infty}<\frac{1}{2}\min_{i\in I(x^*)}\{\frac{\lambda}{|(\nabla f(x^*))_i|},1\}$.
         \end{itemize}
Using the above conclusions, for any $\Delta x \in U $, we can obtain
\begin{eqnarray*}
&H(x^*+\Delta x)-H(x^*)&=\lambda\|x^*+\Delta x\|_0-\lambda\|x^*\|_0+f(x^*+\Delta x)-f(x^*) \\
&& \geq \sum_{i\in I(x^*)}\lambda\|
\Delta x_i\|_0+\langle \nabla f(x^*), \Delta x\rangle\\
&&=\sum_{i\in I(x^*)}(\lambda\|\Delta x_i\|_0+(\nabla f(x^*))_i\Delta x_i)+\sum_{i\not\in I(x^*)}\Delta x_i(\nabla f(x^*))_i\\
&&\geq0
\end{eqnarray*}
So $x^*$ is a local minimizer of objective function $H(x)$.

2. By inequality (\eqref{e3}),
$H(x_{k})$ is non-increasing.  Together with the assumption that $f(x)$ is bounded from below,  we can conclude
$H(x_k)$ is convergent.
Furthermore, $H(x_{k})\rightarrow H(x^*)$ since $I(x_k)=I(x^*)$ when $k\geq k_0$ and $f(x_{k_j})\rightarrow f(x^*)$.
\end{proof}
Note that, by  the proof of Lemma \eqref{lem:nochange} and Theorem \eqref{thm:kexists}, we have the following conclusions which are useful in subsection 3.1 and 3.2.
\begin{itemize}
  \item There exists some $k_0$ such that $I(y_{k})=I(x_k)=I(x^*)$; so we can assume that $x_k=[x^*_k;0],\;\;k\geq k_0$ and $x^*=[x^*_*;0]$ without loss of generality.
  \item When $k\geq k_0$, $\mathcal{H}_{\sqrt{\frac{2{\lambda}}{L+\mu}}} (y_{k}-\dfrac{1}{L+\mu}\nabla f(y_{k}))
=P_{C_*}(y_{k}-\frac{1}{L+\mu}\nabla f(y_{k}))$.
  \item If $x_k=\mathcal{H}_{\sqrt{\frac{2{\lambda}}{L+\mu}}}(x_k-\dfrac{\nabla f(x_k)}{L+\mu})$ for some $k$, $x_k$ ia a local minimizer of $H(x)$. And then VMEPIHT method stops;
  \item $\nabla_* f(x^*)=0$ where $\nabla_* f(x^*)$ denotes the elements of $\nabla f(x^*)$ in linear subspace $C_*$.
\end{itemize}
\subsection{convergence behavior of iterative sequence when $f(x)=\|Ax-b\|^2/2$}
In this subsection, we analyze the convergence behavior of iterative sequence generated by VMEPIHT method when $f(x)=\|Ax-b\|^2/2$.

Firstly,
we explain how to select $H_k$ and step length $\alpha_k$. The strategy we used is very practical. For simplicity, we use the following notations. For a vector $x$, $(x)^k$ denotes the elements of $x$ in linear subspace $C_k$.
Correspondingly, $x_k$ consists of $x^k_k$ and 0 and $y_{k+1}$ also consists of $y^k_{k+1}$ and 0. In linear subspace $C_k$, let $\nabla_k f(x_k)=Q^kx_k^k+c^k$, where $c^k=-(A^Tb)^k$ and the symmetric positive semidefinite matrix $Q^k$ is the elements of $A^TA$ corresponding to $x_k^k$. Clearly, $\nabla_k f(x_k)=(\nabla f(x_k))^k$.

When $f(x)=\|Ax-b\|^2/2$, we obtained $H_k$ and step length $\alpha_k$ through the following ways.
\begin{enumerate}
  \item Symmetric positive definite matrix  $\tilde{H}_k$ is obtained by BFGS or limited-memory BFGS\cite{1999Numerical}. For simplicity, there is no harm in denoting
  $x_k=\left(
               \begin{array}{c}
                 x_k^k \\
                 0 \\
               \end{array}
             \right)$
  and
  \[
 \tilde{ H}_k=\left(
               \begin{array}{cc}
                 H_k^k& * \\
                 * & *\\
               \end{array}
             \right)
  \]
where $H_k^k$ is corresponding to $x^k_k$  and then let
  \[
H_k=\left(
               \begin{array}{cc}
                 H_k^k& 0\\
                 0 & *\\
               \end{array}
             \right)
  \]
  where $*$ represent any positive definite matrix which does not work during the iteration. Then
\[
y_{k+1}=P_{C_k}(x_{k}-\alpha_{k}H_{k}\nabla f(x_{k}))=\left(
               \begin{array}{c}
                 x_k^k+\alpha_k d_k^k\\
                 0 \\
               \end{array}
             \right),\;\;\mbox{where }d_k^k=-H_k^k\nabla_k f(x^{k}).
\]
In practice, we can construct the definite matrix  $\tilde{H}_k$ firstly. Secondly, compute $P_{C_k}(\nabla f(x_k))$ to get
  $\left(
                \begin{array}{c}
                  \nabla_k f(x_{k})\\
                 0 \\
               \end{array}
             \right).$
  Then  naturally
  \[
  y_{k+1}=P_{C_k}(x_k-\alpha_k
\tilde{H}_k \left(
                \begin{array}{c}
                  \nabla_k f(x_{k})\\
                 0 \\
               \end{array}
             \right))=
             \left(
                \begin{array}{c}
                 x^k_k+\alpha_k d_k^k \\
                 0 \\
               \end{array}
             \right).
   \]
  \item If $Q^kd_k^k=0$, let $\alpha_k=0$. Otherwize, it is clear that $d_k^k$ is a descent direction of function $f|_{C_k}$ at point $x_{k}^k$. Let $\alpha_k$ be the optimal steplength, namely
  \[
  \alpha_k=\frac{-\nabla_k f(x_k)^Td_k^k}{d_k^kQ^kd_k^k}=\frac{\nabla_k f(x_k)^TH_k^k\nabla_k f(x_k)}{\nabla_k f(x_k)^TH_k^kQ^kH_k^k\nabla_k f(x_k)}.
  \]
  In practice, the step length
    \[
  \alpha_k=\frac{-\nabla_k f(x_k)^Td_k^k}{d_k^kQ^kd_k^k}=\frac{-\nabla f(x_k)^Td_k}{\|Ad_k\|^2}
  \]
  where $d_k=\left(
               \begin{array}{c}
                 d_k^k \\
                 0 \\
               \end{array}
             \right)$ and hence $\|Ad_k\|^2=((d_k^k)^T, 0)A^TA
    \left(
                \begin{array}{c}
                  d_k^k\\
                 0 \\
               \end{array}
             \right)
  =d_k^kQ^kd_k^k$.
\end{enumerate}

\bigskip
In the following, we will analyze the convergence behavior of $\{x_k\}_{k\geq k_0}$ generated by VMEPIHT method.
Similarly, denote the elements of $H_k$ corresponding to $x_k^*$ by $H_k^*$. Hence $H_k^*$ is a symmetric positive definite matrix. And
\[
f_*(x_k)=f(\left(
               \begin{array}{c}
                 x_k^* \\
                 0 \\
               \end{array}
             \right))=\|A\left(
               \begin{array}{c}
                 x_k^* \\
                 0 \\
               \end{array}
             \right)-b\|^2/2,\;\;
        \nabla_* f(x_k)=(\nabla f(x_k))^*=Qx_k^*+c
        \]
        \[
 d_k^*=-H_k^*\nabla_* f(x^{k}),\;\;       d_k=\left(
               \begin{array}{c}
                 d_k^* \\
                 0 \\
               \end{array}
             \right)
        \]
 where $c=-(A^Tb)^*$ and the symmetric positive semidefinite matrix $Q$ is the elements of $A^TA$ corresponding to $x_k^*$.

If there exists some $k\geq k_0$ such that $\alpha_k=0$, namely $Qd_k^*=0$ and $y_{k+1}=x_k$, combining with the fact $\nabla_* f(x^*)=0$ yields
\[
\langle \nabla_* f(x_k), H_k^*\nabla_* f(x_{k})\rangle=\langle 0-\nabla_* f(x_k), d_k^*\rangle=\langle \nabla_* f(x^*)-\nabla_* f(x_k), d_k^*\rangle=\langle Q(x^*-x_k), d_k^*\rangle=0.
\]
Then $\nabla_* f(x_k)=\nabla_* f(y_{k+1})=0$ right now. Hence
\[
\mathcal{H}_{\sqrt{\frac{2{\lambda}}{L+\mu}}} (y_{k+1}-\frac{1}{L+\mu}\nabla f(y_{k+1}))
=P_{C_*}(x_{k}-\frac{1}{L+\mu}\nabla f(x_{k}))=x_{k}=y_{k+1}.
\]
and so $x_k$ is a local minimizer according to the proof of Theorem \eqref{thm:kexists}. In the following, we assume that VMEPIHT method  gets a infinite set $\{x^k\}$, namely $\alpha_k>0$ for any $ k\geq k_0$.

\begin{theorem}\label{thm:rat}
Let $f(x)=\|Ax-b\|^2/2$ and $\{x_k\}_{k\geq k_0}$ be the sequence generated by VMEPIHT. $x^*$ is any cluster point of $\{x_k\}$ and hence is a local minimizer of $H(x)$.
 Then
\begin{enumerate}
\item $\|x_{k}-x^*\|^2_{A^TA}$ and $\|y_{k}-x^*\|^2_{A^TA}$ tend to 0 decreasingly; if there exists $k\geq k_0$ such that $\|y_{k}-x^*\|^2_{A^TA}=0$, $y_k$ is a local minimizer;
\item if there exist positive constants $m$ and $M$ such that $\prec mI\preceq H_k\prec MI$, then $\|x_{k}-x^*\|^2_{A^TA}$ converges to 0 linearly;
\item if  $Q$ is furthermore a positive definite matrix and $\lim_{k\rightarrow\infty}\dfrac{\|((H_k^*)^{-1}-Q)d_k^*\|}{\|d_k^*\|}=0$,
then $\|x_{k}-x^*\|$ converges to 0 at a superlinear rate.
\end{enumerate}
\end{theorem}
\begin{proof}
For any $k\geq k_0$,  the iterative process of VMEPIHT turns to
\begin{eqnarray*}
&&x_{k}^*=y_{k}^*-\frac{\nabla_* f(y_{k})}{L+\mu}\\
&&y_{k+1}^*=x_{k}^*-\alpha_k H_k^* \nabla_* f(x_{k})\\
&&\alpha_k=\frac{-\nabla_* f(x_k)^Td_k^*}{d_k^*Qd_k^*}=\frac{\nabla_* f(x_k)^TH_k^*\nabla_* f(x_k)}{\nabla_* f(x_k)^TH_k^*QH_k^*\nabla_* f(x_k)}=\frac{(x_k^*-x^*_*)^TQH_k^*Q(x_k^*-x^*_*)}{(x_k^*-x^*_*)^TQH_k^*QH_k^*Q(x_k^*-x^*_*)}.
\end{eqnarray*}
The third equation about $\alpha_k$ holds owing to the fact $\nabla_* f(x^*)=0$.

1. Form the iterative formula above, we have
\begin{equation}\label{pp}
\begin{array}{rl}
&\|y_{k+1}^*-x_*^*\|^2_Q\\
=&\|x_{k}^*-\alpha_k H_k^*\nabla_* f(x_{k})-x_*^*\|^2_Q=\|x_{k}^*-x_*^*-\alpha_k H_k^*(\nabla_* f(x_{k})-\nabla_* f(x^*))\|^2_Q\\
=&\|x_{k}^*-x_*^*\|^2_Q-2\alpha_k\langle x_{k}^*-x_*^*, QH_k^*(\nabla_* f(x_{k})-\nabla_* f(x^*))\rangle+\alpha_k^2\|H_k^*(\nabla_* f(x_{k})-\nabla_* f(x^*))\|^2_Q\\
=&\|x_{k}^*-x_*^*\|^2_Q-\dfrac{\langle x_{k}^*-x_*^*, QH_k^*Q(x_{k}^*-x_*^*)\rangle^2}{(x_k^*-x_*^*)^TQH_k^*QH_k^*Q(x_k^*-x_*^*)}\\
\leq& \|x_{k}^*-x_*^*\|^2_Q\\
\end{array}
\end{equation}

\begin{equation}\label{pp2}
\begin{array}{rl}
\|x_{k}^*-x_*^*\|^2_Q=&\|y_{k}^*-\nabla_* f(y_{k})/(L+\mu)-x_*^*\|^2_Q\\
=&\|y_{k}^*-\nabla_* f(y_{k})/(L+\mu)-x_*^*+\nabla_* f(x^{*})/(L+\mu)\|^2_Q\\
=&\|(I-\frac{Q}{L+\mu})(y_{k}^*-x_*^*)\|^2_Q\\
\leq& \|y_{k}^*-x_*^*\|^2_Q.
\end{array}
\end{equation}
Combining with the fact that $x^*$ is a cluster point of $\{x_k\}$ $x_k-y_k\rightarrow 0$ from Lemma \eqref{lem:nochange} and $\|y_{k}-x^*\|^2_{A^TA}=\|y_{k}^*-x^*_*\|^2_Q$,
we can conclude $\|y_{k}-x^*\|^2_{A^TA}$ and $\|x_{k}-x^*\|^2_{A^TA}$ tend to 0 decreasingly.

If there exists some $k\geq k_0$ such that $\|y^{k}_k-x^*\|^2_{A^TA}=0$, namely $\nabla f(y_k)=\nabla f(x^*)$, then
\[
\mathcal{H}_{\sqrt{\frac{2{\lambda}}{L+\mu}}} (y_k-\frac{\nabla f(y_k)}{L+\mu})=\mathcal{H}_{\sqrt{\frac{2{\lambda}}{L+\mu}}} (y_k-\frac{\nabla f(x^*)}{L+\mu})\ni y_k.
\]
Hence  $y_k$ is a local minimizer according to the proof of Theorem \eqref{thm:kexists} and algorithm stops.

2. If $0\prec mI\prec H_k\prec MI$, similarly with the procedure \eqref{pp} and \eqref{pp2}, we have
\begin{equation}\label{ee3}
\renewcommand\arraystretch{1.3}
\begin{array}{rl}
\|x_{k+1}^*-x_*^*\|^2_Q=&\|(I-\dfrac{Q}{L+\mu})(y_{k+1}^*-x_*^*)\|^2_Q\\
\leq&\|y_{k+1}^*-x_*^*\|^2_Q\\
=&\|x_{k}^*-x_*^*\|^2_Q-\dfrac{\langle x_{k}^*-x_*^*, QH_k^*Q(x_{k}^*-x_*^*)\rangle}{(x_k^*-x_*^*)^TQH_k^*QH_k^*Q(x_k^*-x_*^*)}\langle x_{k}^*-x_*^*, QH_k^*Q(x_{k}^*-x_*^*)\rangle\\
\leq& \|x_{k}^*-x_*^*\|^2_Q-\dfrac{m}{ML}\langle x_{k}^*-x_*^*, QQ(x_{k}^*-x_*^*)\rangle\\
\leq&(1-\dfrac{m\lambda_t}{ML}) \|x_{k}^*-x_*^*\|^2_Q
\end{array}
\end{equation}
where $\lambda_t$ is $Q$'s minimum eigenvalue except 0 and the last inequality holds due to Lemma \eqref{Q}. Naturally $\|x_{k}-x^*\|^2_{A^TA}=\|x_{k}^*-x_*^*\|^2_Q$ converges to 0 linearly.

3. If  $Q$ is a positive definite matrix and $\lim_{k\rightarrow\infty}\dfrac{\|((H_k^*)^{-1}-Q)d_k^*\|}{\|d_k^*\|}=0$, we have
\begin{equation}\label{alpha}
\alpha_k-1=\frac{\langle(H_k^*)^{-1}d_k^*, d_k^*\rangle}{d_k^*Qd_k^*}-1=\frac{\langle d^k, ((H_k^*)^{-1}-Q)d_k^*\rangle}{\langle d_k^*, Q d_k^*\rangle}\rightarrow0.
\end{equation}
Combining $Q(y_{k+1}^*-x_k^*)=\nabla_*f(y_{k+1}^*)-\nabla_*f(x_{k}^*)$ with the iteration procedure $y_{k+1}^*=x_{k}^*-\alpha_k H_k^* \nabla_* f(x_{k})$, we can obtain
\begin{eqnarray*}
\frac{\|\nabla_*f(y_{k+1}^*)\|}{\|y_{k+1}^*-x_k^*\|}&=&
\frac{\|(Q-\alpha_k^{-1}(H_k^*)^{-1})d_{k}^*\|}{\|d_{k}^*\|}\\
&=&\frac{\|\alpha_k^{-1}(Q-(H_k^*)^{-1})d_{k}^*+(1-\alpha_k^{-1})Qd_{k}^*\|}{\|d_{k}^*\|}\\
&\leq&\frac{\|\alpha_k^{-1}(Q-(H_k^*)^{-1})d_{k}^*\|+\|(1-\alpha_k^{-1})Ld_{k}^*\|}{\|d_{k}^*\|}
\end{eqnarray*}
Let $\lambda_{\min}$ is the minimum eigenvalue of $Q$. It follows from \eqref{alpha} and $\lim_{k\rightarrow\infty}\dfrac{\|((H_k^*)^{-1}-Q)d_k^*\|}{\|d_k^*\|}=0$ that
\[
0\leftarrow\frac{\|\nabla_*f(y_{k+1}^*)\|}{\|y_{k+1}^*-x_k^*\|}=\frac{\|\nabla_*f(y_{k+1}^*)-\nabla_*f(x^*)\|}{\|y_{k+1}^*-x_k^*\|}\geq\frac{\lambda_{\min}\|y_{k+1}^*-x^*_*\|}{\|y_{k+1}^*-x_*^*\|+\|x_{k}^*-x_*^*\|},
\]
which, together with the conclusion $\|x_{k+1}^*-x_*^*\|_Q\leq\|y_{k+1}^*-x_*^*\|_Q$,  imply
\[
\lim_{k\rightarrow\infty}\dfrac{\|x_{k+1}^*-x^*_*\|}{\|x_{k}^*-x_*^*\|}=0.
\]
Thus $\|x_{k}-x^*\|$ converges to 0 at a superlinear rate.
\end{proof}
\begin{remark}
If symmetric positive definite matrix  $\tilde{H}_k$ is obtained by limited-memory BFGS by the set of vector pairs $\{S_i,Y_i\}_{i=k-T}^{k-1}$ where $S_{k-1}=x_{k}-y_k$,
$S_{k-2}=y_{k}-x_{k-1}$, $\cdots $ and $Y_i=A^TAS_i+tI,\;\;t>0$.
So it's clear that $\tilde{H}_k$ have positive upper and lower bounds. And the same for $H_k$.
\end{remark}
\begin{remark}
The Theorem \eqref{thm:rat} tells us a superlinear convergence rate can be attained if $(H_k^*)^{-1}$ become accurate approximation to $Q$ along the direction $d_k^*$.
Although we can't prove that it holds. It worth believing that the VMEPIHT method is fast.
\end{remark}

\subsection{convergence behavior of iterative sequence for general convex differentiable $f(x)$}
In this subsection, we will analyze the convergence behavior of iterative sequence for general convex differentiable $f(x)$.
Under attainable conditions, we prove the  iterative sequence's convergence.

We assume that the sequence of symmetric positive definite matrix $\{H_k\}$  and steplength satisfy the following conditions.
\begin{enumerate}
\item $H_k=\left(
               \begin{array}{cc}
                 H_k^*& 0\\
                 0 & *\\
               \end{array}
             \right)$
  when $k\geq k_0$ and there exists $k_1$ such that $(H_{k+1}^*)^{-1}\preceq(1+\epsilon_k)(H_k^*)^{-1},\;\;$ where $\epsilon_k\geq0$ and $\sum_{k=\max\{k_0,k_1\}}^\infty\epsilon_k<+\infty.$
  \item If $\nabla_kf(x_{k}^k)^Td_k^k=0$, let $\alpha_k=0$; otherwise, $\alpha_k\in \{ \frac{2\|\nabla_* f(x_{k})\|^2}{L\|\nabla_* f(x_{k})\|^2_{H_k^*}}\gamma^i ,\;\;i=0,1,2,\cdots\}$ satisfies Dong's step length rule\cite{2010dong} which can guarantee the function value of $f$ to decrease , namely
  \[
  \langle\nabla_kf(\left(
               \begin{array}{c}
                 y_{k+1}^k\\
                 0 \\
               \end{array}
             \right)),d_k^k\rangle\geq\delta\langle\nabla_k f(\left(
               \begin{array}{c}
                 x_{k}^k\\
                 0 \\
               \end{array}
             \right)),d_k^k\rangle
  \]where $\gamma,\delta\in(0,1)$.
\end{enumerate}

\begin{lemma}\label{ss}\cite{1987Introduction}
Let $\{a_k\}$ and $\{\eta_k\}$ be nonnegative sequences of real nmubers such that
\[
a_{k+1}\leq(1+\eta_k)a_k,\;\; \sum_{k=1}^\infty\eta_k<+\infty.
\]
Then $\{a_k\}$ converges.
\end{lemma}

\smallskip
\begin{theorem}\label{thm:con}
Assume that $f(x)$ satisfies Assumption A and is furthermore twice continuously differentiable.
Let $\{x_k\}_{k\geq k_0}$ be the sequence generated by VMEPIHT. $x^*$ is any cluster point of $\{x_k\}$ and hence is a local minimizer of $H(x)$.
 Then $ \|x_{k}^*-x_*^*\|_{(H_{k-1}^*)^{-1}}\rightarrow0$.
\end{theorem}
\begin{proof}
For any $k\geq \max\{k_0,k_1\}$,  the iterative process of VMEPIHT turns to
\begin{eqnarray*}
&&x_{k}^*=y_{k}^*-\frac{\nabla_* f(y_{k})}{L+\mu}\\
&&y_{k+1}^*=x_{k}^*-\alpha_k H_k^* \nabla_* f(x_{k})
\end{eqnarray*}
Form the iterative formula above, we have
\begin{eqnarray*}
\|y_{k+1}^*-x_*^*\|^2_{(H_k^*)^{-1}}&=&\|x_{k}^*-\alpha_k H_k^*\nabla_* f(x_{k})-x_*^*\|^2_{(H_k^*)^{-1}}\\
&=&\|x_{k}^*-x_*^*\|^2_{(H_k^*)^{-1}}-2\alpha_k\langle x_{k}^*-x_*^*, \nabla_* f(x_{k})-\nabla_* f(x^*))\rangle+\alpha_k^2\|\nabla_* f(x_{k})\|^2_{H_k^*}\\
&=&\|x_{k}^*-x_*^*\|^2_{(H_k^*)^{-1}}-\frac{2\alpha_k}{L}\|\nabla_* f(x_{k})-\nabla_* f(x^*)\|^2+\alpha_k^2\|\nabla_* f(x_{k})\|^2_{H_k^*}\\
&=&\|x_{k}^*-x_*^*\|^2_{(H_k^*)^{-1}}-\frac{2\alpha_k}{L}\|\nabla_* f(x_{k})\|^2+\alpha_k^2\|\nabla_* f(x_{k})\|^2_{H_k^*}\\
&\leq&\|x_{k}^*-x_*^*\|^2_{(H_k^*)^{-1}}\\
&=&\|y_{k}^*-\nabla_* f(y_{k})/(L+\mu)-x_*^*\|^2_{(H_k^*)^{-1}}\\
&=&\|y_{k}^*-\nabla_* f([y_{k}^*;0])/(L+\mu)-x_*^*+\nabla_* f([x^{*}_*;0])/(L+\mu)\|^2_{(H_K^*)^{-1}}\\
&=&\|(I-\frac{\nabla^2_*f(\xi_k)}{L+\mu})(y_{k}^*-x_*^*)\|^2_{(H_k^*)^{-1}}\\
&\leq& (1+\epsilon_{k-1})\|y_{k}^*-x_*^*\|^2_{(H_{k-1}^*)^{-1}}.
\end{eqnarray*}
Hence $\|y_{k}^*-x_*^*\|^2_{(H_{k-1}^*)^{-1}}$ converges according to Lemma \eqref{ss}. Then $\|y_{k}^*-x_*^*\|^2_{(H_{k-1}^*)^{-1}}\rightarrow0$.
\end{proof}

\begin{remark}
In practice, we can obtain $\{H_k\}_{k\geq k_0}$ in the similar way as subsection 3.1. Firstly construct the definite matrix  $\tilde{H}_k$ by BFGS or limited-memory BFGS method. Secondly, compute $P_{C_k}(\nabla f(x_k))$. So get
  $\left(
                \begin{array}{c}
                  \nabla_k f(x_{k})\\
                 0 \\
               \end{array}
             \right).$
and
\[ y_{k+1}=P_{C_k}(x_k-\alpha_k
\tilde{H}_k \left(
                \begin{array}{c}
                  \nabla_k f(x_{k})\\
                 0 \\
               \end{array}
             \right))=
             \left(
                \begin{array}{c}
                 x^k_k+\alpha_k d_k^k \\
                 0 \\
               \end{array}
             \right).
             \]
Finally, for some $k_1$, let $\tilde{H}_k=\tilde{H}_t$ for any $ k\geq k_1$.  Once $k\geq \max\{k_0,k_1\}$,
\[
 y_{k+1}=P_{C_*}(x_k-\alpha_k
\tilde{H}_t \left(
                \begin{array}{c}
                  \nabla_* f(x_{k})\\
                 0 \\
               \end{array}
             \right))=
             \left(
                \begin{array}{c}
                 x^k_*+\alpha_k \tilde{H}_t^*  \nabla_* f(x_{k})\\
                 0 \\
               \end{array}
             \right).
 \]
Hence $H_k^*=\tilde{H}_{k_1}^* $ for any $k\geq \max\{k_0,k_1\}$. Naturally, the conditions $(H_{k+1}^*)^{-1}\preceq(1+\epsilon_k)(H_k^*)^{-1},\;\;k\geq k_1$ and $\sum_{k=k_1}^\infty\epsilon_k<+\infty$ hold.
\end{remark}

\section{Numerical results}
\label{sec:test}In this section, we present the performance of VMEPIHT method for solving $\ell_0$ minimization problems through some numerical tests on compressive sensing and MRI Imaging. The solvers, nmAP\eqref{nmAPG}, niAPG\eqref{niAPG} and nPIHT\eqref{nPIHT} are compared. We do not compare with the variable metric type method as we do not know the effective strategy of selecting metric matrix. All the experiments are conducted in MATLAB using a desktop computer equipped with a $.0$GHz $8$-core i7 processor and $16$GB memory.

\subsection{Compressive sensing}
For this experiment, we  consider the following $ \ell_0$ regularization compressive sensing  problem coming from \cite{Yuling2017Iterative}
\begin{equation}\label{exsig2}
\min_{x\in\bR^n}\frac{1}{2}\|Ax-b\|^2+\lambda\|x\|_0.
\end{equation}
The data matrix $A\in \bR^{m\times n}$ ($n\in\{10000, 18000\},m=\lfloor\frac{n}{4}\rfloor$) is a  random Gaussian or random Bernoulli sensing matrix and the columns are normalized to have $\ell_2$ norm of $1$. The true signals $x^*$ are chosen $\lfloor\frac{m}{32}\rfloor$-sparse. The observed data $b\in \bR^m$ is generated by
\[
b=Ax^*+\eta,
\]
where $\eta$ is a white Gaussian noise of variance $0.02$. Since the optimal regularization parameter $\lambda^*$ depends on noise level $\|\eta\|$ which is unknown in practice, we predefine a path
\[\{\lambda=\|A^Tb\|^2_\infty\exp^{t}, t=\log1:\dfrac{\log(10^{-10})-\log1}{199}:\log(10^{-10})\}\]
and select the optimal $\lambda^*$.

  For all methods, the initial guess is $x_0=A^Tb$. $\tilde{x}$ is a local minimizer when $\tilde{x}\in\mathcal{H}_{\sqrt{\frac{2{\lambda}}{L+\mu}}} (\tilde{x}-\frac{1}{L+\mu}\nabla f(\tilde{x})).$ So we use the following stopping criteria
\[
\frac{\|x_{k+1}-a_k\|}{\max\{1,\|x_k\|\}}<1\times10^{-5}.
\]
where $a_k$ represents $x_k$ or $y_k$ for nmAPG method, $v_k$ for niAPG method, $y_{k+1}$ for nPIHT method and our VMEPIHT  method.
 The parameters of methods are showed in Table \eqref{PM} where metric matrix is updated by limited-memory BFGS by vector pairs with $T$ members for our method.
\begin{table}
\centering
\begin{tabular}{|c|c|c|c|c|}
\hline
Method&nmAP\eqref{nmAPG}&niAPG\eqref{niAPG}&nPIHT\eqref{nPIHT}&VMEPINT\\
\hline
\multirow{2}*{parameters}&$\alpha_x=\alpha_y=1/(L+10^{-6})$&$\eta=1/(L+10^{-6})$&$\mu=10^{-6}$&$\mu=10^{-6}$\\
&$\eta=0.1$& $q=2$&$\omega_k=0.9999$&$T=6$\\
\hline
\end{tabular}
\caption{The parameters of methods for CS.   }\label{PM}
\end{table}

For each algorithm and each choice of $n$, we conduct 50 experiments and record the CPU runtime, the relative error $\dfrac{\|x-x^*\|}{\|x^*\|}$ and the iteration number. Related results are shown in Figure \eqref{CS1} and Figure \eqref{CS2}.
In addition, for all the methods, $I(\tilde{x})$ is always equal to $I(x^*)$ and we do not list them.
For all methods, the relative error is almost the same which means that the quality of approximate solution is comparability. Our method, VMEPIHT method have obvious runs significantly faster compared with other three method.

\begin{figure}[!htp]
\centering
\begin{tabular}{c@{\hspace{0pt}}c@{\hspace{0pt}}c@{\hspace{0pt}}c}
\includegraphics[width=0.34\textwidth]{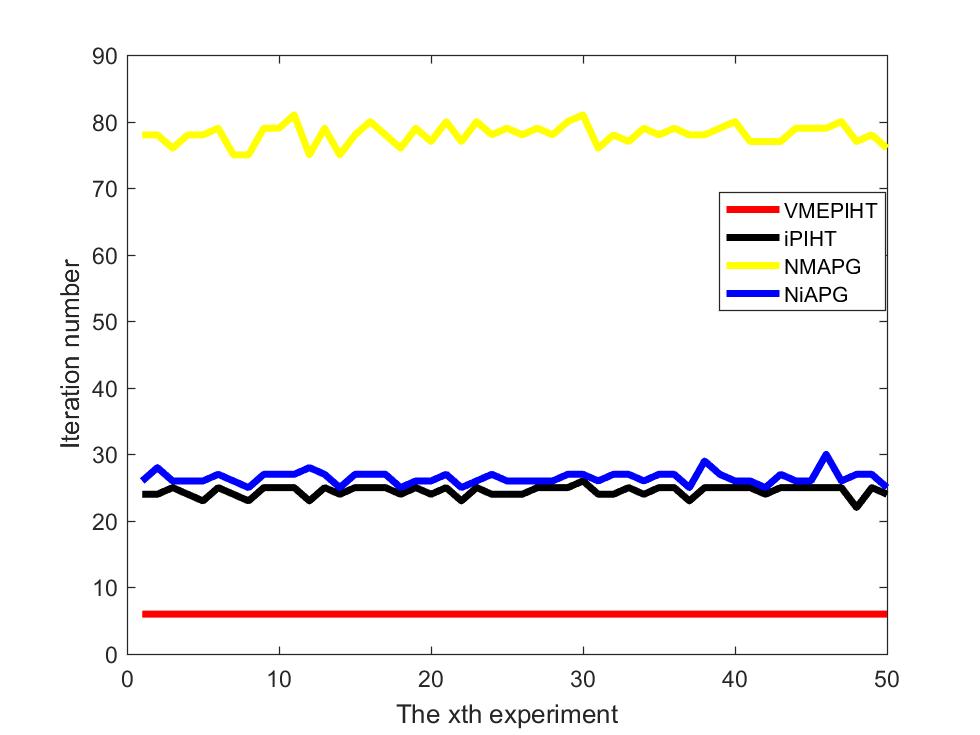} &\includegraphics[width=0.34\textwidth]{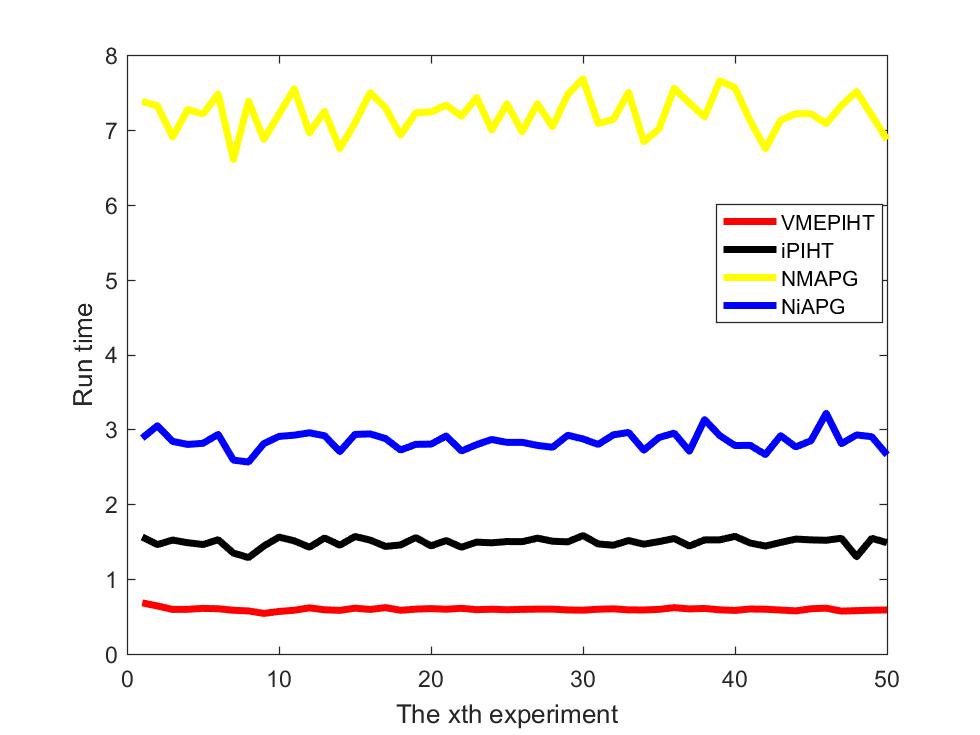}
&\includegraphics[width=0.34\textwidth]{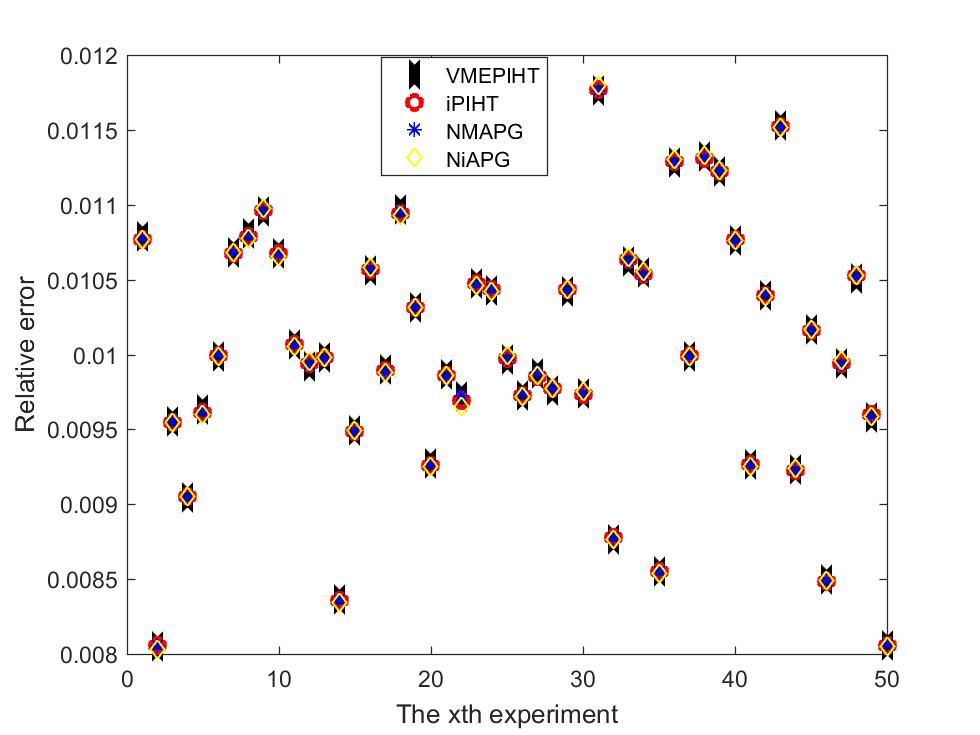}\\
\includegraphics[width=0.34\textwidth]{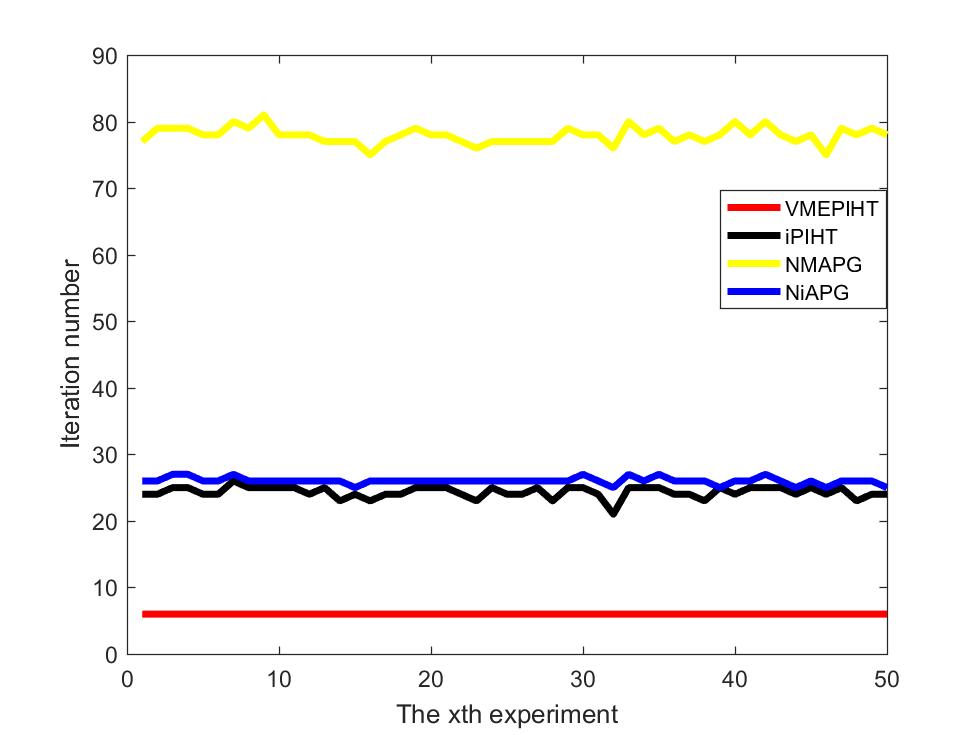}
&
\includegraphics[width=0.34\textwidth]{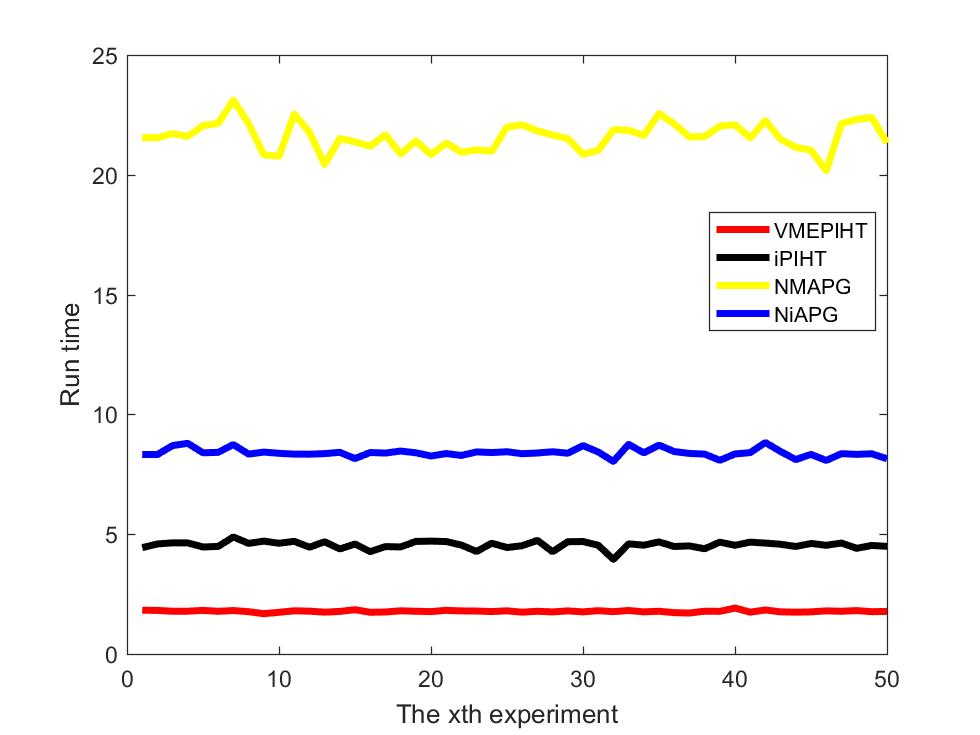}
 &
\includegraphics[width=0.34\textwidth]{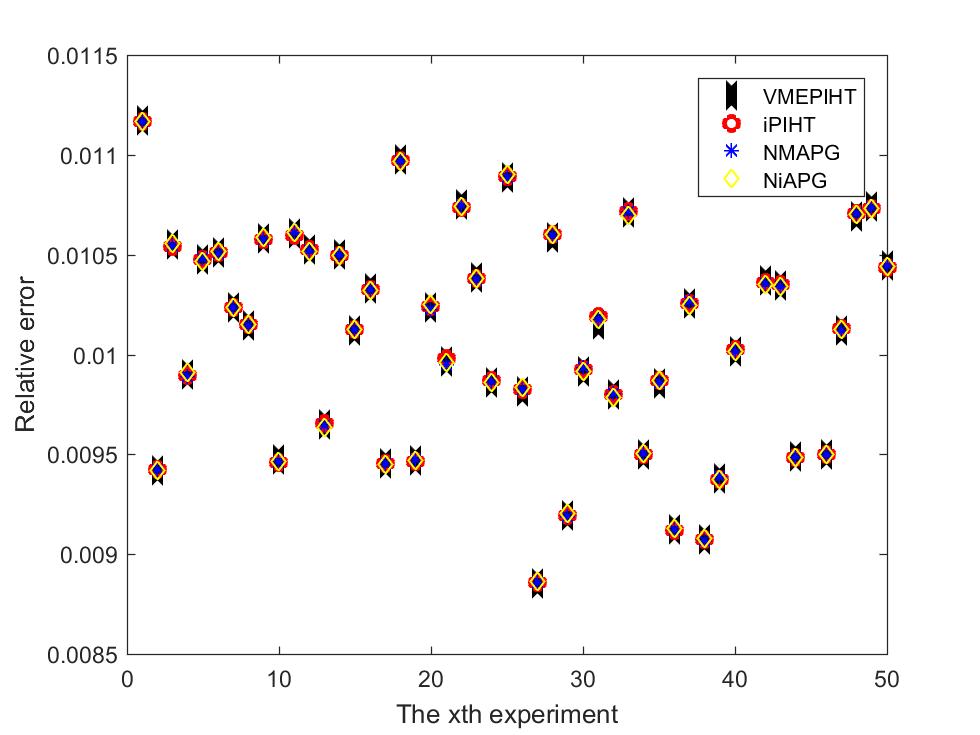} \\
\end{tabular}
\caption{Numerical results(iterative number, CPU time and relative error) with random Gaussian sensing matrix $A$ of size $n=10000$ (left) and $n=18000$ (right). }
\label{CS1}
\end{figure}

\begin{figure}[!htp]
\centering
\begin{tabular}{c@{\hspace{0pt}}c@{\hspace{0pt}}c@{\hspace{0pt}}c}
\includegraphics[width=0.34\textwidth]{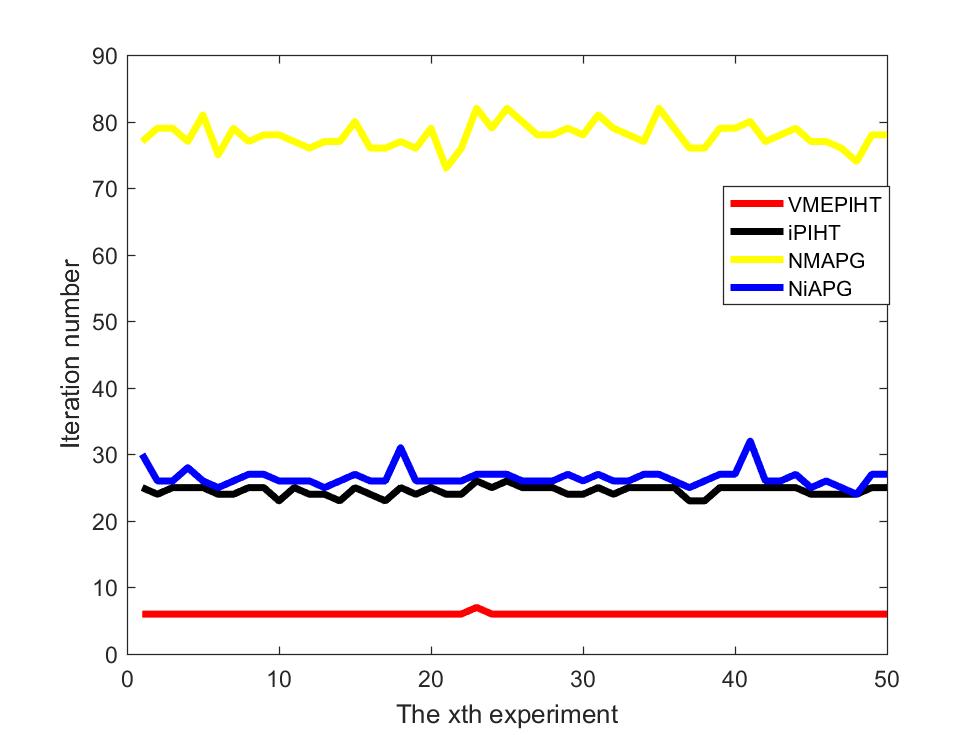} &\includegraphics[width=0.34\textwidth]{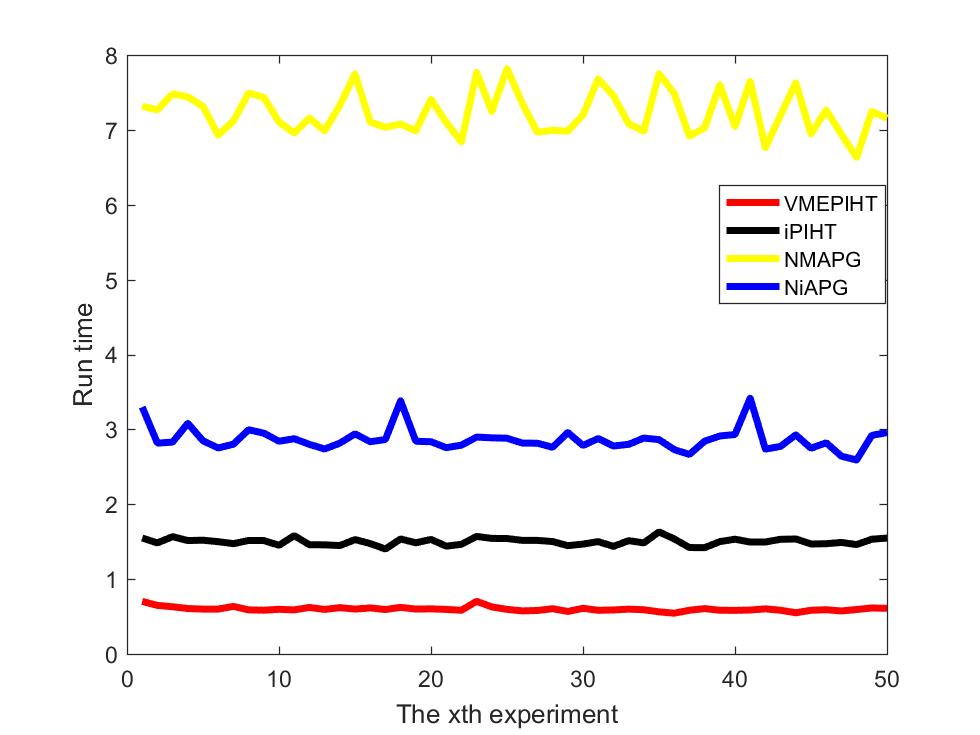}
&\includegraphics[width=0.34\textwidth]{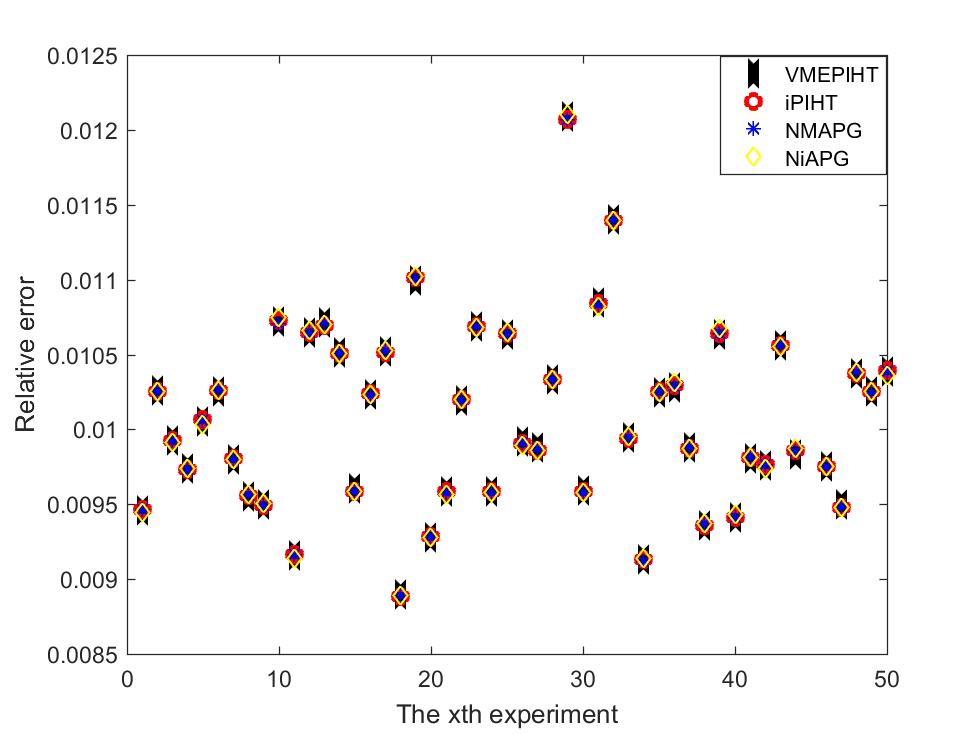}\\
\includegraphics[width=0.34\textwidth]{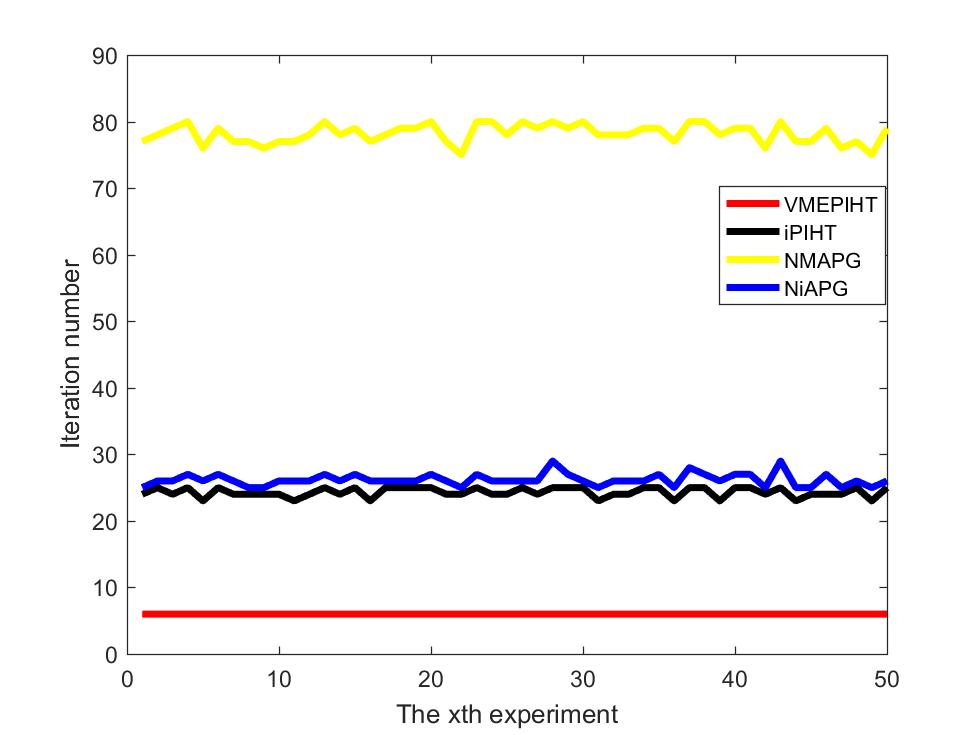}
&
\includegraphics[width=0.34\textwidth]{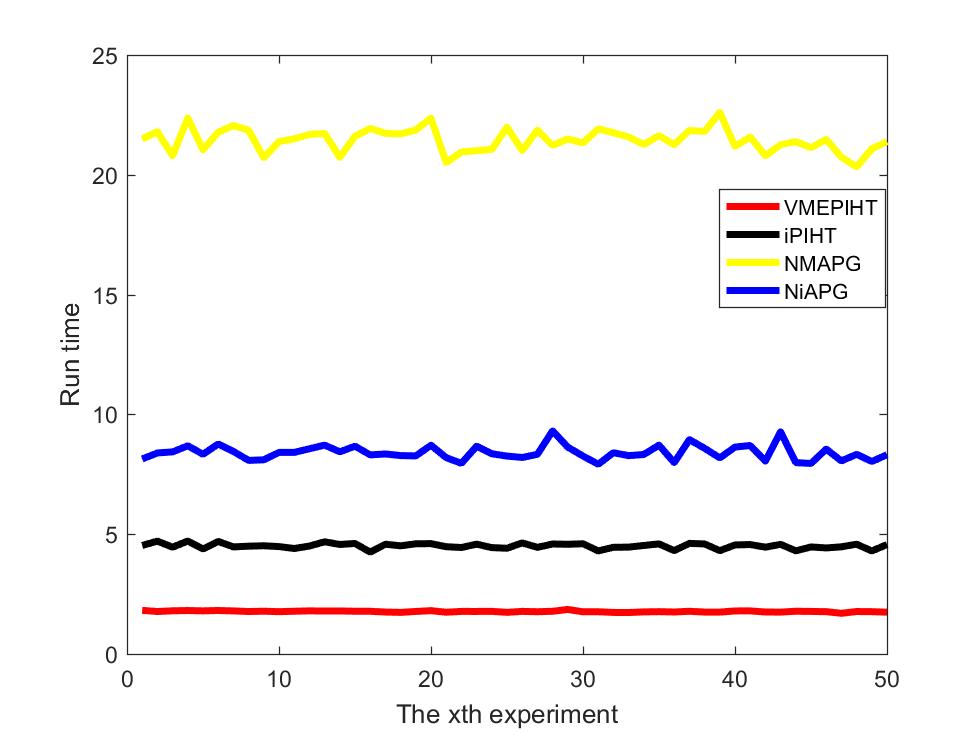}
 &
\includegraphics[width=0.34\textwidth]{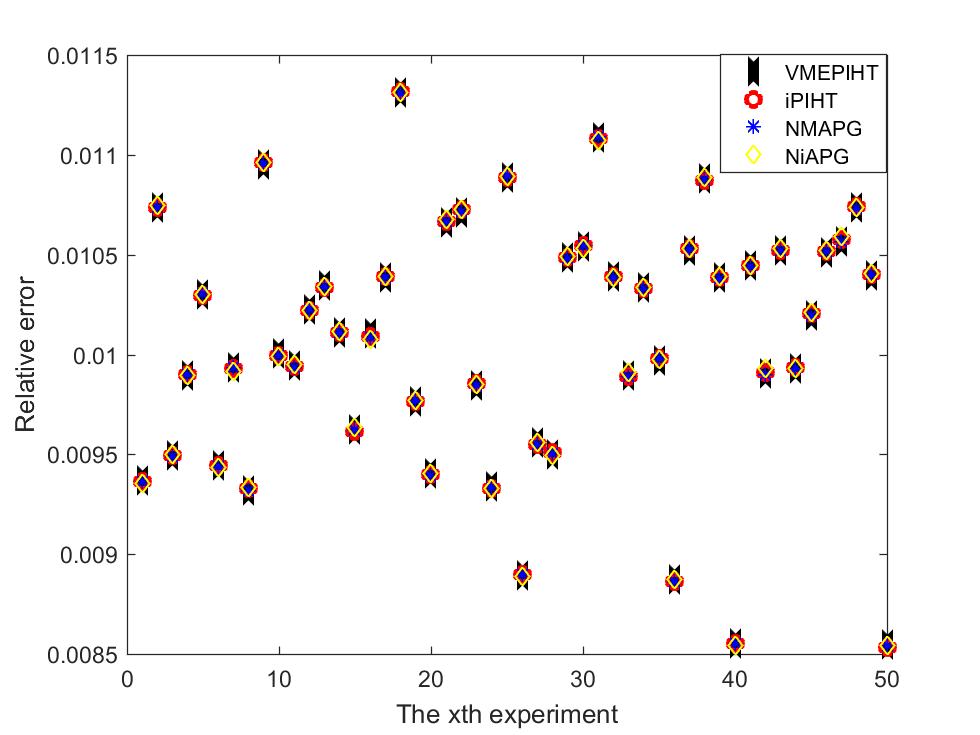} \\
\end{tabular}
\caption{Numerical results(iterative number, CPU time and relative error) with random Bernoulli sensing matrix $A$ of size $n=10000$ (left) and $n=18000$ (right). }
\label{CS2}
\end{figure}

\subsection{CT image reconstruction}
In this subsection, we  use the following $ \ell_0$ regularization  problem to reconstructed CT images,
\begin{equation}
\min_{x} \frac{1}{2}\|AW^T x-b\|^2+\frac{\kappa}{2}\|(I-WW^T)x\|^2+\lambda\|x\|_0,
\label{eq:bal}
\end{equation}
For this experiment, the test image is Shepp-Logan phantom with size $128\times 128$ generated by MATLAB built-in functions, degradation matrix $A$ happens to be a projection matrix with $50$ projections, the wavelet tight frame transform $W$ is the $2D$ tensor-product piecewise linear spline wavelet tight frame system whose level of wavelet decomposition is set to be $4$.  After the construction of the projection matrix $A$, we then add some Gaussian noise with variance $\sigma=0.01\cdot\|Af\|_\infty$ to the projected data to obtain the observed data $b$ where $f$ represents the true image.

For all methods, the stopping criteria is
\[
\frac{\|x_{k+1}-a_k\|}{\max\{1,\|x_k\|\}}<\mbox{tol}.
\]
where $a_k$ represents $x_k$ or $y_k$ for nmAPG method, $v_k$ for niAPG method, $y_{k+1}$ for nPIHT method and our VMEPIHT  method.
 The parameters of methods are showed in Table \eqref{CTP}.
\begin{table}
\centering
\begin{tabular}{|c|c|c|c|c|}
\hline
Method&nmAP\eqref{nmAPG}&niAPG\eqref{niAPG}&nPIHT\eqref{nPIHT}&VMEPINT\\
\hline
\multirow{2}*{parameters}&$\alpha_x=\alpha_y=1/(L+10^{-6})$&$\eta=1/(L+10^{-6})$&$\mu=10^{-6}$&$\mu=10^{-6}$\\
&$\eta=0.4$& $q=9$&$\omega_k=0.9999$&$T=2$\\
\hline
\end{tabular}
\caption{The parameters of methods for CT.   }\label{CTP}
\end{table}

For different choice of tol, we conduct experiments and record the CPU runtime, the PSNR values and the iteration number. Related results are fully summarized in Table \eqref{CTT}.
VMEPIHT method always takes less processing iterative number and have high PSNR. However, it's CPU time don't have absolute advantages. The longer processing time of VMEPIHT when $\mbox{tol}=5\times10^{-5}$ is largely due to the computation of $H_k\nabla f(x_k)$. On the whole, VMEPIHT method have an advantage over other three method.

\begin{table}
\centering
\begin{tabular}{|p{1.5cm}|c|c|c|c|}
\hline
\diagbox[width=1.9cm]{\footnotesize{tol}}{\footnotesize{Method}}&nmAP\eqref{nmAPG}&niAPG\eqref{niAPG}&nPIHT\eqref{nPIHT}&VMEPIHT\\
\hline
$5\times10^{-5}$ &$145.73/231/31.5533$ &$201.80/247/31.4024$  &$\textbf{98.39}/262/31.2354$ &107.34/\textbf{116}/\textbf{32.0719}\\
 \hline
 $1\times10^{-5}$ &$246.39/407/31.5682$  &$229.95/288/31.3905$  &$118.27/313/31.3362$  & $\textbf{115.16}/\textbf{129}/\textbf{32.1153}$\\
\hline
$5\times10^{-6}$  &$278.67/464/31.5906$ &$269.50/339/31.3605$ &$150.55/404/31.3944$ &$\textbf{143.03}/\textbf{152}/\textbf{32.1409}$\\
\hline
$1\times10^{-6}$  &$549.05/836/31.5799$  &$305.22/386/31.3585$ &$213.67/580/31.4364$   &$\textbf{178.39}/\textbf{198}/\textbf{32.1817}$\\
\hline
\end{tabular}
\caption{Numerical results(CPU time/iterative number/PSNR) for CT.  }\label{CTT}
\end{table}
\section{Conclusions and perspectives}
In this paper, we proposed  a variable metric extrapolation proximal iterative hard thresholding (VMEPIHT) method which combines the PIHT method and the skill in quasi-newton method for solving $\ell_0$ regularized minimization problem. We provide some convergence results for the proposed methed, including the convergence of  iterative sequence, linear convergence rate and superlinear convergence rate. Although nmAPG and niAPG
is able to solve more general nonconvex problem compared with VMEPIHT method. VMEPIHT method which designed by using $\|x\|_0$'s characteristics, has better numerical performance on CS and CT reconstruction problems.

\section*{Acknowledgements}

The work were partially supported by NSFC(11901368).

\section*{References}

\bibliography{ref}   

@article{ZhangDong13,
  title={$\ell_{0}$ Minimization for wavelet frame based image restoration},
  author={Y Zhang, B Dong, Z Lu}.
  journal={Mathematics of Computation},
  year={2012},
  volume={82},
 number={282},
  pages={995-1015},
}

@article{2019Wavelet,
  title={Wavelet frame-based image restoration using sparsity, nonlocal, and support prior of frame coefficients},
  author={L He,  Y Wang, Z Xiang },
  journal={The Visual Computer},
  volume={35},
  number={2},
  pages={151-174},
  year={2019},
}

@article{Zengli2018,
  title={An adaptive iteration reconstruction method for limited-angle CT image reconstruction},
  author={C Wang,  L Zeng,  L Zhang, Y Guo, W Yu},
  journal={Journal of Inverse and Ill-posed Problems},
   volume={26},
  number={6},
  year={2018},
}

@article{Machine,
  title={A Multiobjective Sparse Feature Learning Model for Deep Neural Networks},
  author={M Gong, J Liu, H Li, Q  Cai, L Su },
  journal={IEEE Transactions on Neural Networks and Learning Systems},
  volume={26},
  number={12},
  pages={3263-3277},
  year={2017},
}

@article{2013Sparse,
  title={Sparse Coding From a Bayesian Perspective},
  author={X Lu, Y  Wang, Y Yuan},
  journal={IEEE Transactions on Neural Networks \& Learning Systems},
  volume={24},
  number={6},
  pages={929-939},
  year={2013},
}

@article{2020Iterative,
  title={Iterative Potts Minimization for the Recovery of Signals with Discontinuities from Indirect Measurements: The Multivariate Case},
  author={ L Kiefer, M Storath, A  Weinmann},
  journal={Foundations of Computational Mathematics},
   volume={21},
  pages={649-694},
  year={2021},
}

@article{2017A,
  title={A reconstruction algorithm for electrical capacitance tomography via total variation and l0-norm regularizations using experimental data},
  author={J Chen, M Zhang, Y Li},
  year={2017},
}

@ARTICLE{Lu12,
   author = {Z Lu},
  title = {Iterative hard thresholding methods for $\ell_0$ regularized convex
	cone programming},
  journal = {Mathematical Programming},
   volume={147},
  number={1-2},
  pages={125-154},
  year={2013},
}

@ARTICLE{FB79,
   author = {Lions P.L., Mercier B.},
  title = {Splitting algorithms for the sum of two nonlinear operators},
  journal = {SIAM journal on numerical analysis},
   volume={16},
   number={6},
  pages={964-979},
  year={1979},
}

@article{1979Ergodic,
  title={Ergodic convergence to a zero of the sum of monotone operators in Hilbert space},
  author={G Passty},
  journal={Journal of Mathematical Analysis and Applications},
  volume={72},
  number={2},
  pages={383-390},
  year={1979},
}

@article{2013Convergence,
  title={Convergence of descent methods for semi-algebraic and tame problems: proximal algorithms, forward-backward splitting, and regularized Gauss-Seidel methods},
  author={H Attouch, J Bolte , BF Svaiter},
  journal={Mathematical Programming},
  volume={137},
  number={1-2},
  pages={91-129},
  year={2013},
}

@article{bot2014inertial,
 title={An inertial forward¨Cbackward algorithm for the minimization of the sum of two nonconvex functions},
  author={RI Bot, ER Csetnek, SL L\'{a}szl\'{o}},
  journal={EURO Journal on Computational Optimization},
  volume={4},
  number={1},
  pages={3-25},
  year={2016},
}

@article{chubulai,
title={Accelerated Proximal Gradient Methods for Nonconvex Programming},
author={H Li, Z Lin},
journal={NIPS: Proceedings of the 28th International Conference on Neural Information Processing Systems},
volume={1},
pages={379-387},
year={2015},
}

@article{2018Inexact,
  title={Efficient Inexact Proximal Gradient Algorithm for Nonconvex Problems},
  author={Q Yao, JT  Kwok, F Gao, W Chen, TY Liu},
  journal={Proceedings of the Twenty-Sixth International Joint Conference on Artificial Intelligence}
  year={2017},
}

@article{2019A,
  title={A New Proximal Iterative Hard Thresholding Method with Extrapolation for $\ell_0$ Minimization},
  author={X Zhang, X Zhang},
  journal={Journal of Scientific Computing},
  volume={79},
  number={2},
  pages={809-826},
  year={2019},
}

@article{2016A,
  title={A Multi-step Inertial Forward¨CBackward Splitting Method for Non-convex Optimization},
  author={ J Liang, J Fadili, G Peyr¨¦},
  year={2016},
}

@article{BONETTINI2016A,
  title={A Variable Metric Forward-Backward Method with Extrapolation},
  author={S Bonettini, F Porta, V Ruggiero},
  journal={SIAM Journal on Scientificuting},
volume={38},
  number={4},
  pages={A2558¨CA2584},
  year={2016},
}

@article{2015Splitting,
  title={Splitting Methods with Variable Metric for Kurdyka¨Cojasiewicz Functions and General Convergence Rates},
  author={P Frankel, G Garrigos, J  Peypouquet},
  journal={Journal of Optimization Theory and Applications},
  volume={165},
  number={3},
  pages={874-900},
  year={2015},
}

@article{2016On,
  title={On the convergence of variable metric line-search based proximal-gradient method under the Kurdyka-Lojasiewicz inequality},
  author={S Bonettini, I Loris, F Porta, M Prato, S Rebegoldi},
  journal={Inverse Problems},
  volume={33},
  number={5},
  year={2016},
}

@article{2016The,
  title={The variable metric forward-backward splitting algorithm under mild differentiability assumptions},
  author={S Salzo},
  journal={Siam Journal on Optimization},
  volume={27},
  number={4},
  pages={2153-2181},
  year={2017},
}

@article{2013A,
  title={A block coordinate variable metric forward-backward algorithm},
  author={E Chouzenoux, JC Pesquet, A Repetti },
  journal={Journal of Global Optimization},
  volume={66},
  number={3},
  pages={457-485},
  year={2016},
}

@article{2020Variable,
  title={Variable metric techniques for forward¨Cbackward methods in imaging},
  author={S Bonettini,  F Porta, V Ruggiero, L Zanni },
  journal={Journal of Computational and Applied Mathematics},
  volume={385},
  pages={113192},
  year={2021},
}

@article{2016Unifying,
  title={Unifying abstract inexact convergence theorems and block coordinate variable metric iPiano},
 journal={SIAM Journal on Optimization},
  author={P Ochs},
    volume={29},
  number={1},
  pages={541-570},
  year={2019},
}

@ARTICLE{BT09,
  author = {A Beck, M Teboulle},
  title = {A fast iterative shrinkage-thresholding algorithm for linear inverse
	problems},
  journal = {SIAM Journal on Imaging Sciences},
    volume = {2},
  number = {1},
   pages = {183--202},
  year = {2009},
}

@book{1987Introduction,
  title={Introduction to optimization},
  author={ B Polyak},
  publisher={Chapman and Hall},
  year={1987},
}

@article{2009A,
  title={A scaled gradient projection method for constrained image deblurring},
  author={S  Bonettini, R Zanella, L Zanni},
  journal={Inverse Problems},
  volume={25},
  pages={015002},
  year={2009},
}

@article{2014New,
  title={New convergence results for the scaled gradient projection method},
  author={S Bonettini, M Prato},
  journal={Inverse Problems},
  volume={31},
  number={9},
  pages={095008},
  year={2015},
}

@article{2015A,
  title={A New Steplength Selection for Scaled Gradient Methods with Application to Image Deblurring},
  author={F Porta, M Prato, L  Zanni},
  journal={Journal of Scientific Computing},
  volume={65},
  number={3},
  pages={895-919},
  year={2015},
}

@article{2015Zhang,
  title={A Note on the Complexity of Proximal Iterative Hard Thresholding Algorithm},
  author={ X Zhang, X Zhang},
  journal={Journal of the Operations Research Society of China},
  volume={3},
  number={4},
  pages={459-473},
  year={2015},
}

@article{2010dong,
  title={New step lengths in conjugate gradient methods},
  author={Y  Dong},
  journal={Computers and Mathematics with Applications},
  volume={60},
  number={ 3},
  pages={563-571},
  year={2010},
}

@book{1999Numerical,
  title={Numerical Optimization Second Edition},
  author={J Nocedal,SJ Wright},
  publisher={World Scientific},
  year={1999},
}

@article{Yuling2017Iterative,
  title={Iterative Soft/Hard Thresholding With Homotopy Continuation for Sparse Recovery},
  author={Y Jiao, B Jin, X Lu},
  journal={IEEE Signal Processing Letters},
    volume={24},
  number={6},
  pages={784-788},
  year={2017},
}

\end{document}